\documentclass[regno]{amsart}
\usepackage{eurosym}
\usepackage{latexsym,amsmath,amssymb,amscd}
\usepackage[all]{xy}
\usepackage{tikz-cd}
\usepackage{enumerate}
\usepackage{hyperref}
\usepackage{mathrsfs}
\usepackage{amsfonts}
\usepackage{setspace}
\usepackage{mathtools}

\makeatletter
\@addtoreset{figure}{section}
\def\thefigure{\thesection.\@arabic\c@figure}
\def\fps@figure{h,t}
\@addtoreset{table}{bsection}
\def\thetable{\thesection.\@arabic\c@table}
\def\fps@table{h, t}
\@addtoreset{equation}{section}

\makeatother

\newtheorem{theorem}{Theorem}
\newtheorem{corollary}[theorem]{Corollary}
\newtheorem{definition}[theorem]{Definition}
\newtheorem{example}[theorem]{Example}

\newtheorem{lemma}[theorem]{Lemma}

\newtheorem{proposition}[theorem]{Proposition}
\newtheorem{remark}[theorem]{Remark}
\newtheorem{construction}[theorem]{Construction}

\numberwithin{theorem}{section}
\numberwithin{equation}{section}

\email{maria.joita@upb.ro and mjoita@fmi.unibuc.ro }
\urladdr{http://sites.google.com/a/g.unibuc.ro/maria-joita}
\email{ionutsimon.gh@gmail.com}
\begin{document}
\title[Shilov boundary ideal]{On the Shilov boundary ideal for Fr\'{e}chet
local operator systems}
\author[1]{Maria Joi\c ta}
\address{Department of Mathematics, Faculty of Applied Sciences, University
Politehnica of Bucharest, 313 Spl. Independentei, 060042, Bucharest, Romania}
\author[2]{Gheorghe-Ionu\c{t} \c{S}imon}
\address{Department of Mathematics, Faculty of Applied Sciences, University
Politehnica of Bucharest, 313 Spl. Independentei, 060042, Bucharest, Romania 
}

\begin{abstract}
We show that the Shilov boundary ideal for a separable Fr\'{e}chet local
operator system is given by the intersection of the kernels of all its $%
\Gamma$-boundary representations.
\end{abstract}

\maketitle

\section{Introduction}

Locally $C^*$-algebras are generalizations of the $C^*$-algebras where the
topology on a locally $C^*$-algebra is defined by a separating family of $%
C^* $-seminorms instead of a $C^*$-norm. They are also known in the
literature as $LMC^*$-algebras, $b^*$-algebras, or pro-$C^*$-algebras. The
term locally $C^*$-algebra was introduced by A. Inoue \cite{In72} in the
early 1970s to address problems in functional analysis and quantum physics
where norm-boundedness fails, but local $C^*$-type behaviour persists. We
refer the reader to the book of Fragoulopoulou \cite{Fr05} for a historical
review of the theory, which goes back to 1952. Many concepts and results
from the theory of $C^*$-algebras have been extended in the locally convex
setting. One such concept is that of an operator space.

Effros and Webster \cite{EW97} initiated the study of locally convex
analogues of operator spaces, known as local operator spaces. In 2008, A.
Dosiev \cite{Do08} provided a concrete realization of local operator spaces
as subspaces of the locally $C^*$-algebra $C^*_{\mathcal{E}}(\mathcal{D})$
of unbounded operators on a quantized domain $\lbrace \mathcal{H};\mathcal{E}%
; \mathcal{D} \rbrace$, thereby generalizing Ruan's representation theorem
for operator spaces (see, e.g. \cite[Theorem 2.3.5]{ER00}). In some sense,
the locally $C^*$-algebra $C^*_{\mathcal{E}}(\mathcal{D})$ plays the role of 
$B(\mathcal{H})$ from the theory of $C^*$-algebras. A representation of a
locally $C^*$-algebra $\mathcal{A}$ on a quantized domain $\lbrace \mathcal{H%
}; \mathcal{E}; \mathcal{D} \rbrace$ is a continuous $*$-homomorphism from $%
\mathcal{A}$ to $C^*_{\mathcal{E}}(\mathcal{D})$. A local operator system is
a unital self-adjoint subspace of a locally $C^*$-algebra. In \cite{Ark22},
Arunkumar introduced the notion of a local boundary representation for a
local operator system as a locally convex analogue of boundary
representations for operator systems. As pointed out in \cite{Jo23}, the
study initiated by Arunkumar was shown to reduce, in fact, to the study of
representations on Hilbert spaces. In \cite{Jo23}, \cite{Jo25} the notion of
a local boundary representation on Hilbert spaces was introduced for local
operator systems, and several of its fundamental properties were
investigated. In the same work, the concept of the Shilov boundary ideal for
a local operator system was developed, and it was proved that, for a
separable local operator system, the Shilov boundary ideal exists and is
given by the intersection of the kernels of all $p_{\lambda}$-boundary
representations for the local operator system.

Building on the results of \cite{Simon}, we introduce the notion of a $%
\Gamma $-boundary representation for a local operator system and prove that
the Shilov boundary ideal for a separable Fr\'{e}chet local operator system $%
\mathcal{S}$ is given by the intersection of the kernels of all $\Gamma$%
-boundary representations of $\mathcal{S}$ (see Theorem \ref{Theorem}).
Furthermore, we present an explicit example of a Shilov boundary ideal for a
separable Fr\'{e}chet local operator system within this framework.

\section{Preliminaries}

\subsection{Locally $C^{\ast }$-algebras}

Let $\mathcal{A}$ be a $\ast $-algebra with unit, denoted by $1_{\mathcal{A}
} $. A seminorm $p$ on $\mathcal{A}$ is called \textit{sub-multiplicative}
if $p(1_{\mathcal{A}})=1$ and $p(ab)\leq p(a)p(b)$ for all $a,b\in \mathcal{%
A }$. A sub-multiplicative seminorm $p$ on $\mathcal{A}$ is called a \textit{%
$C^{\ast}$-seminorm} if $p(a^{\ast }a)=p(a)^{2}$ for all $a\in \mathcal{A}.$

Let $\left( \Lambda ,\leq \right) $ be a directed poset and let $\lbrace
p_{\lambda }\rbrace_{\lambda \in \Lambda }$ be a family of $C^{\ast}$
-seminorms defined on some $\ast $-algebra $\mathcal{A}$. We say that $%
\lbrace p_{\lambda }\rbrace_{\lambda \in \Lambda }$ is an \textit{upward
filtered family} of $C^{\ast}$-seminorms if $p_{\lambda _{1}}(a)\leq
p_{\lambda _{2}}(a)\ $ for all $\ a\in \mathcal{A}\ $ whenever $\ \lambda
_{1}\leq \lambda _{2}\ $in$\ \Lambda $.

A \textit{locally $C^*$-algebra} $\mathcal{A}$ is a complete Hausdorff
topological $*$-algebra over $\mathbb{C}$ whose topology is determined by a
family of continuous $C^*$-seminorms $\lbrace p_{\lambda}
\rbrace_{\lambda\in\Lambda}$. 

If the family of $C^*$-seminorms defining the topology on a locally $C^*$%
-algebra is countable, we say that it is a Fr\'{e}chet locally $C^*$%
-algebra. Moreover, the topology on a Fr\'{e}chet locally $C^*$-algebra is
unique.

Morphisms between locally $C^*$-alegbras are continuous $*$-homomorphisms.
Note that in contrast to the theory of $C^*$-algebras, $*$-homomorphisms of
locally $C^*$-algebras might not be continuous. However, $*$-homomorphisms
between Fr\'{e}chet locally $C^*$-algebras are automatically continuous \cite%
[Theorem 5.2]{Ph88}. Any $C^*$-algebra is a locally $C^*$-algebra in the
above sense. \medskip

An element $a\in \mathcal{A}$ is called \textit{local self-adjoint} if $%
a=a^{\ast}+c$ for some $c\in \mathcal{A}$ with $p_{\lambda }\left( c\right)
=0$ for some $\lambda \in \Lambda ,$ and we call $a$ as $\lambda $
-self-adjoint, and \textit{local positive} if $a=b^{\ast}b+c$ where $b,c\in 
\mathcal{A}$ and $p_{\lambda }\left( c\right) =0$ for some $\lambda \in
\Lambda $; we call $a$ as $\lambda $-positive and write $a\geq _{\lambda }0$
. We write $a=_{\lambda }0$ whenever $p_{\lambda }\left( a\right) =0$.

Note that an element $a\in \mathcal{A}$ is self-adjoint if and only if $a$
is $\lambda $-self-adjoint for all $\lambda \in \Lambda \ $and $a$ is
positive if and only if $a$ is $\lambda $-positive for all $\lambda \in
\Lambda .$

\subsection{Local completely positive maps}

Let $\mathcal{A}$ and $\mathcal{B}$ be two unital Fr\'{e}chet locally $%
C^{\ast}$-algebras whose topologies are defined by the families of $C^{\ast}$%
-seminorms $\lbrace p_{m }\rbrace_{m\geq 1}$ and $\lbrace
q_{l}\rbrace_{l\geq 1 }$, respectively. For each $n\in\mathbb{N}, \ M_{n}(%
\mathcal{A})$ denotes the set of all $n\times n$ matrices over $\mathcal{A}$%
. Note that $M_{n}(\mathcal{A})$ is a unital Fr\'{e}chet locally $C^{\ast}$
-algebra with the associated family of $C^{\ast}$-seminorms $\lbrace p_{m
}^{n}\rbrace_{m\geq 1 }$.

Let $\varphi :\mathcal{A}\rightarrow \mathcal{B}$ be a linear map. We say
that $\varphi $ is \textit{local completely positive} if for each $l\geq 1$,
there exists $m\geq 1$ such that $[\varphi (a_{ij})]\geq _{l}0$ whenever $%
[a_{ij}]\geq _{m}0$ and $[\varphi (a_{ij})]=_{l}0$ if $[a_{ij}]=_{m}0$, for
all $n\in \mathbb{N}.$

A local completely positive map $\varphi :\mathcal{A}\rightarrow \mathcal{B} 
$ is \textit{local completely isometric} if 
\begin{equation*}
q_{m}^{n}\left( [\varphi (a_{ij})]_{i,j=1}^{n}\right)
=p_{m}^{n}([a_{ij}]_{i,j=1}^{n})
\end{equation*}%
for all $[a_{ij}]_{i,j=1}^{n}\in M_{n}(\mathcal{A})$, for all $n\geq 1$ and
for all $m\geq 1.$

Clearly, if $\varphi :\mathcal{A}\rightarrow \mathcal{B}$ is a local
isometric $\ast $-homomorphism, then $\varphi $ is a local completely
isometric map.

A \textit{local operator system} is a self-adjoint subspace $\mathcal{S}$ of
a unital locally $C^{\ast }$-algebra $\mathcal{A}$ which contains the unit
of $\mathcal{A}$. If $\varphi :\mathcal{S}_{1}\rightarrow \mathcal{S}_{2}$
is a local completely positive map, then $\varphi (s^{\ast })=\varphi
(s)^{\ast },$ and so $\varphi (\mathcal{S}_{1})$ is a local operator system.


\subsection{The unbounded analogue of the Gelfand-Naimark theorem}

In this subsection, we recall some notions about the algebra of all linear
operators on quantized domains in a Hilbert space from \cite{Do08} and \cite%
{Do12}, denoted by $C^*_{\mathcal{E}}(\mathcal{D})$. This algebra is
considered the natural analogue of $B(\mathcal{H})$ in the local operator
space theory. However, it should be noted that this algebra works
differently than $B(\mathcal{H})$ in the operator space theory, and more
careful consideration should be given. \medskip

Let $\left( \Omega, \leq \right)$ be a directed poset. A \textit{quantized
domain} in a Hilbert space $\mathcal{H}$ is a triple $\left\lbrace \mathcal{H%
};\mathcal{E};\mathcal{D} \right\rbrace$, where $\mathcal{E}:=\left\lbrace 
\mathcal{H}_{\iota}: \iota\in\Omega \right\rbrace$ is an upward filtered
family of closed subspaces such that the union space $\mathcal{D}%
:=\bigcup\limits_{\iota\in\Omega}\mathcal{H}_{\iota}$ is dense in $\mathcal{H%
}$. A quantized domain $\mathcal{E}$ is called a \textit{Fr\'{e}chet
quantized domain} if $\mathcal{E}$ is a countable family.

Note that the quantized family $\mathcal{E}:=\left\lbrace \mathcal{H}%
_{\iota}: \iota\in\Omega \right\rbrace$ determines an upward filtered family 
$\left\lbrace P_{\iota}: \iota\in\Omega \right\rbrace$ of projections in $B(%
\mathcal{H})$, where $P_{\iota}$ is the orthogonal projection of $\mathcal{H}
$ onto the closed subspace $\mathcal{H}_{\iota}$. If $\left\lbrace
P_{\iota}: \iota\in\Omega \right\rbrace$ consists of mutually commuting
projections, then we say that $\mathcal{E}$ is a \textit{commutative domain}
in $\mathcal{H}$. Any Fr\'{e}chet quantized domain is a commutative domain.

Let $\left\{ \mathcal{H},\mathcal{E},\mathcal{D}\right\} $ be a quantized
domain. The set

\begin{equation*}
\begin{aligned} C_{\mathcal{E}}^{*}(\mathcal{D}) = \{\, T \in L(\mathcal{D})
\mid\,& T(\mathcal{H}_{\iota}) \subseteq \mathcal{H}_{\iota}, \;
T(\mathcal{H}_{\iota}^{\perp} \cap \mathcal{D}) \subseteq
\mathcal{H}_{\iota}^{\perp} \cap \mathcal{D}, \\ & T|_{\mathcal{H}_{\iota}}
\in B(\mathcal{H}_{\iota}), \ \forall\, \iota \in \Omega \,\}, \end{aligned}
\end{equation*}%
where $L(\mathcal{D})$ denotes the set of all linear operators on $\mathcal{D%
}$, is a locally $C^{\ast }$-algebra with the involution 
\begin{equation*}
T^{\ast }:=T^{\bigstar }\upharpoonright _{\mathcal{D}}\in C_{\mathcal{E}%
}^{\ast }(\mathcal{D}),
\end{equation*}%
and the topology induced by the family of $C^{\ast }$-seminorms $\{\Vert
\cdot \Vert _{\iota }\}_{\iota \in \Omega }$, where 
\begin{equation*}
\Vert T\Vert _{\iota }:=\Vert T\upharpoonright _{\mathcal{H}_{\iota }}\Vert
_{B(\mathcal{H}_{\iota })}.
\end{equation*}

Let $C^*_{\mathcal{E}}(\mathcal{D})$ be a Fr\'{e}chet locally $C^*$-algebra,
and let $\lbrace P_{m} \rbrace_{m\geq 1}$ be the projection sequence
associated to $\mathcal{E}$. We denote by $S_{m}:=\left( I-P_{m-1}
\right)P_{m}$ the projection onto the subspace $\mathcal{H}_{m-1}^{\perp}\cap%
\mathcal{H}_{m},\ m\geq 1$, where for $m=1$ we set $S_{1}=P_{1}$. If $T\in
C^*_{\mathcal{E}}(\mathcal{D})$, then it has a diagonal representation $%
T=\sum\limits_{k=1}^{\infty}S_{k}TS_{k}$.

Let $\mathcal{A}$ be a unital locally $C^*$-algebra $\mathcal{A}$ with the
topology defined by the family of $C^*$-seminorms $\lbrace p_{\lambda}
\rbrace_{\lambda\in\Lambda}$. A \textit{local representation} of $\mathcal{A}
$ on a quantized domain $\lbrace \mathcal{H};\mathcal{E};\mathcal{D} \rbrace$
with $\mathcal{E}=\lbrace \mathcal{H}_{\iota} \rbrace_{\iota\in\Omega}$ is a 
$*$-homomorphism $\pi:\mathcal{A}\rightarrow C^*_{\mathcal{E}}(\mathcal{D})$
with the property that for each $\iota\in\Omega$, there exist $%
\lambda\in\lambda$ such that $\Vert \pi(a) \Vert_{\iota}\leq p_{\lambda}(a)$
for all $a\in\mathcal{A}$. If $\Lambda=\Omega$ and for all $%
\lambda\in\Lambda $ we have $\Vert \pi(a) \Vert_{\lambda}=p_{\lambda}(a)$,
for all $a\in\mathcal{A}$, we say that $\pi$ is a \textit{local isometric
representation}.

\begin{theorem}[\protect\cite{Do08}, \protect\cite{Do12}]
For any locally $C^*$-algebra $\mathcal{A}$, there exist a commutative
domain $\lbrace \mathcal{H};\mathcal{E};\mathcal{D} \rbrace$ and a local
isometric $*$-homomorphism $\pi:\mathcal{A}\rightarrow C^*_{\mathcal{E}}(%
\mathcal{D}).$
\end{theorem}

\section{The Shilov boundary ideal for a Fr\'{e}chet local operator system}

Let $\mathcal{A}$ be a Fr\'{e}chet locally $C^*$-algebra whose topology is
defined by the family of $C^*$-seminorms $\lbrace p_{m} \rbrace_{m\geq 1}$.
Suppose that $\mathcal{I}$ is a closed two-sided $*$-ideal of $\mathcal{A}$.
Then the quotient $*$-algebra $\mathcal{A}/\mathcal{I}$ is a Fr\'{e}chet
locally $C^*$-algebra with respect to the family of $C^*$-seminorms $\lbrace 
\widehat{p_m} \rbrace_{m\geq 1}$, where $\widehat{p_{m}}(a+\mathcal{I}%
):=\inf\lbrace p_{m}(a+b)\mid b\in\mathcal{I} \rbrace,\ a\in\mathcal{A}$
(see \cite{Fr05}).

\subsection{Boundary ideal for a Fr\'{e}chet local operator system}

Let $\mathcal{S}\subseteq C^*_{\mathcal{E}}(\mathcal{D})$ be a Fr\'{e}chel
local operator system and let $\mathcal{A}$ be the Fr\'{e}chet locally $C^*$%
-algebra generated by $\mathcal{S}$

\begin{definition}
A closed two-sided $*$-ideal $\mathcal{I}$ of $\mathcal{A}$ is called a 
\textit{boundary ideal for $\mathcal{S}$} if the canonical map $\sigma_{%
\mathcal{I}}:\mathcal{A}\rightarrow \mathcal{A}/\mathcal{I}$ is local
completely isometric on $\mathcal{S}$. A boundary ideal for $\mathcal{S}$ is
called \textit{the Shilov boundary ideal for $\mathcal{S}$} if it contains
every other boundary ideal.
\end{definition}

Let $\mathcal{S}_{0}$ be an operator system, and let $\mathcal{A}_{0}:=C^*(%
\mathcal{S}_{0})$ denote the unital $C^*$-algebra generated by $\mathcal{S}%
_{0}$.

Let $\lbrace X_{m};\ \iota_{mn}:X_{m}\hookrightarrow X_{n}, m\leq n \rbrace$
be an inductive system of compact Hausdorff spaces, and let $%
X=\varinjlim\limits_{m}X_{m}$ denote its inductive limit, endowed with the
corresponding inductive limit topology. Then $C\left( X,\mathcal{A}_{0}
\right):=\lbrace f:X\rightarrow\mathcal{A}_{0}\mid f\ \mathit{is continuous}
\rbrace$ is a unital Fr\'{e}chet locally $C^*$-algebra with respect to the
topology given by the family of $C^*$-seminorms $\lbrace p_{m}
\rbrace_{m\geq 1}$ , where $p_{m}(f):=\sup\lbrace \Vert f(x) \Vert_{\mathcal{%
A}_{0}}\mid x\in X_{m} \rbrace$. Then $S:=\lbrace f\in C\left( X,\mathcal{A}%
_{0} \right)\mid f(x)\in\mathcal{S}_{0}, (\forall)\ x\in X \rbrace$ is a
local operator system, and $C^*(\mathcal{S})=C\left( X,\mathcal{A}_{0}
\right)$, since $C\left( X,\mathcal{A}_{0} \right)\cong C(X)\otimes\mathcal{A%
}_{0}$ (see \cite[Section 3]{Ph88} and the discussion in \cite{BhKa91}).

\begin{proposition}
\label{PP} Let $\mathcal{I}_{0}\subseteq\mathcal{A}_{0}$ be the boundary
ideal for $\mathcal{S}_{0}$. Then $\mathcal{I}:=C\left( X,\mathcal{I}_{0}
\right)$ is a boundary ideal for $\mathcal{S}$.
\end{proposition}

\begin{proof}
We have to show that the map 
\begin{equation*}
Q:C\left( X,\mathcal{A}_{0} \right)\rightarrow C\left( X,\mathcal{A}_{0}
\right)/C\left( X,\mathcal{I}_{0} \right), f\mapsto f+C\left( X,\mathcal{I}%
_{0} \right)
\end{equation*}
is local completely isometric on $\mathcal{S}$. So, we have to prove that 
\begin{equation*}
\widehat{p}_{m}^{n}\left( [Q(f_{ij})] \right)=p_{m}^{n}\left( [f_{ij}]
\right), (\forall)\ [f_{ij}]\in M_{n}(\mathcal{S}), (\forall)\ m\geq 1,
(\forall)\ n\geq 1.
\end{equation*}

Since $Q$ is a morphism of locally $C^*$-algebras, we have that 
\begin{equation*}
\widehat{p}_{m}^{n}\left( [Q(f_{ij})] \right)\leq p_{m}^{n}\left( [f_{ij}]
\right), (\forall)\ [f_{ij}]\in M_{n}(\mathcal{S}), (\forall)\ m\geq 1,
(\forall)\ n\geq 1.
\end{equation*}

Conversely, let $[f_{ij}]\in M_{n}(\mathcal{S})$ and $[g_{ij}]\in M_{n}(%
\mathcal{I}), m\geq 1$. Then

\begin{align*}
p_{m}^{n}\!\left( [f_{ij}] + [g_{ij}] \right) &= \sup \Big\{ \big\| \lbrack
f_{ij}(x)] + [g_{ij}(x)] \big\|_{M_{n}(\mathcal{A}_{0})} \;\Big|\; x \in
X_{m} \Big\} \\
&\geq \big\| \lbrack f_{ij}(x)] + [g_{ij}(x)] \big\|_{M_{n}(\mathcal{A}%
_{0})}, \ (\forall)\, x \in X_{m}.
\end{align*}

Let $x_{0}\in X_{m}$. Then

\begin{align*}
\widehat{p}_{m}^{n}\!\left( [Q(f_{ij})] \right) &= \inf \bigl\{\,
p_{m}^{n}\!\left( [f_{ij}] + [g_{ij}] \right) \,\big|\, [g_{ij}] \in M_{n}(%
\mathcal{I}) \bigr\} \\[0.4em]
&\geq \inf \Bigl\{ \bigl\| \lbrack f_{ij}(x_{0})] + [g_{ij}(x_{0})] \bigr\|%
_{M_{n}(\mathcal{A}_{0})} \,\Big|\, [g_{ij}] \in M_{n}(\mathcal{I}) \Bigr\}
\\[0.4em]
\intertext{\qquad \qquad \qquad \qquad (since $\mathcal{I}_{0}$ is a
boundary ideal for $\mathcal{S}$)}
&\geq \inf \Bigl\{ \bigl\| \lbrack f_{ij}(x_{0})] + [c_{ij}] \bigr\|_{M_{n}(%
\mathcal{A}_{0})} \,\Big|\, [c_{ij}] \in M_{n}(\mathcal{I}_{0}) \Bigr\} \\
&=\Vert [f_{ij}(x_{0})] \Vert_{M_{n}(\mathcal{A}_{0})}.
\end{align*}

Therefore, 
\begin{equation*}
\widehat{p}_{m}^{n}\!\left( [Q(f_{ij})] \right)\geq\sup\lbrace \Vert
[f_{ij}(x)] \Vert_{M_{n}(\mathcal{A}_{0})}\mid x\in X_{m}
\rbrace=p_{m}^{n}\left( [f_{ij}] \right).
\end{equation*}
\end{proof}

\begin{remark}
\label{Rem} Let $\mathcal{S}_{0}$ be an operator system, and let $\mathcal{A}%
_{0}:=C^{\ast }(\mathcal{S}_{0})$ denote the unital $C^{\ast }$-algebra
generated by $\mathcal{S}_{0}$. Let $\mathcal{J}_{0}$ be the Shilov boundary
ideal for $\mathcal{S}_{0}$. Then $C\left( \mathbb{N},\mathcal{A}_{0}\right) 
$, the $\ast $-algebra of continuous functions from $\mathbb{N}$ into $%
\mathcal{A}_{0}$ (with $\mathbb{N}$ endowed with the discrete topology), is
a locally $C^{\ast }$-algebra with respect to the family of $C^{\ast }$%
-seminorms $\{q_{n}\}_{n\in \mathbb{N}}$, where $q_{n}(f)=\Vert f(n)\Vert $.
Moreover, $\mathcal{S}=C\left( \mathbb{N},\mathcal{S}_{0}\right) $ is a
local operator system. By Proposition \ref{PP}, $\mathcal{J}=C\left( \mathbb{%
N},\mathcal{J}_{0}\right) $ is a boundary ideal for $\mathcal{S}$.
Furthermore, $\mathcal{J}$ is the Shilov boundary ideal for $\mathcal{S}$.
\medskip

For this, let $\mathcal{I}\subseteq C\left( \mathbb{N}, \mathcal{A}_{0}
\right)$ be an arbitrary boundary ideal for $\mathcal{S}$. Fix $n_{0}\geq 1$%
, and define $\mathcal{I}_{0}:=\lbrace g(n_{0})\mid g\in\mathcal{I} \rbrace$%
. Then $\mathcal{I}_{n_{0}}$ is a closed two-sided $*$-ideal of $\mathcal{A}%
_{0}$. For each $a_{0}\in \mathcal{A}_{0}$, consider the constant map $%
f_{a_{0}}\in C(\mathbb{N}, \mathcal{A}_{0})$ defined by $f_{a_{0}}(n)=a_{0}$
for all $n\in\mathbb{N}$. Using the fact that $\mathcal{I}$ is a boundary
ideal for $\mathcal{S}$, one obtains, for every $[s_{ij}]\in M_{n}(\mathcal{S%
}_{0})$,

\begin{align*}
\big\| \lbrack s_{ij}] \big\|_{M_{n}(\mathcal{A}_{0})} &=
q_{n_{0}}^{n}\!\left( [f_{s_{ij}}] \right)= \widehat{q}_{n_{0}}^{n}\!\left(
[Q(f_{s_{ij}})] \right) \\
&= \inf \Big\{\, q_{n_{0}}^{n}\!\left( [f_{s_{ij}}] + [g_{ij}] \right) \ %
\Big|\ [g_{ij}] \in M_{n}(\mathcal{I}) \,\Big\} \\
&= \inf \Big\{\, \big\| \lbrack s_{ij}] + [g_{ij}(n_{0})] \big\|_{M_{n}(%
\mathcal{A}_{0})} \ \Big|\ [g_{ij}] \in M_{n}(\mathcal{I}) \,\Big\} \\
&= \inf \Big\{\, \big\| \lbrack s_{ij}] + [c_{ij}] \big\|_{M_{n}(\mathcal{A}%
_{0})} \ \Big|\ [c_{ij}] \in M_{n}(\mathcal{I}_{n_{0}}) \,\Big\} \\
&= \big\| \lbrack s_{ij}] \big\|_{M_{n}(\mathcal{A}_{0}/\mathcal{I}%
_{n_{0}})} .
\end{align*}

This shows that $\mathcal{I}_{n_{0}}$ is a boundary ideal for $\mathcal{S}%
_{0}$. Since $\mathcal{J}_{0}$ is the Shilov boundary ideal for $\mathcal{S}%
_{0}$, it follows that $\mathcal{I}_{n_{0}} \subseteq \mathcal{J}_{0}$ for
every $n_{0} \in \mathbb{N}$, and hence $\mathcal{I} \subseteq C\!\left( 
\mathbb{N}, \mathcal{J}_{0} \right) = \mathcal{J}$.
\end{remark}

\medskip

Determining the boundary representations of an operator system is, in
general, a difficult problem. Although abstract results guarantee their
existence, the boundary property is characterized by a delicate uniqueness
condition for completely positive extensions, which is highly sensitive to
dilation phenomena and to the irreducible representation theory of the
generated $C^*$-algebra. Moreover, boundary representations lack a geometric
characterization analogous to the commutative case, and their structure
depends intricately on how the operator system is embedded in its $C^*$%
-envelope. As a result, few classes of operator systems admit a complete and
accessible description of their boundary representations. An illustrative
example is provided by Argerami and Farenick \cite{ArFa13}, who determine
the boundary representations and the $C^*$-envelope of an irreducible
periodic weighted unilateral shift, highlighting both the subtlety of the
problem and the scarcity of explicit examples. More precisely, they prove
that the Shilov boundary ideal for the operator system $\mathcal{S}_{0}$
generated by an irreducible periodic unilateral weighted shift $W\in B\left(
l^2(\mathbb{N}) \right)$ with period $p$ is $K\left( l^2(\mathbb{N}) \right)$%
, the $C^*$-algebra of compact operators acting on $l^2(\mathbb{N})$.
\medskip

\begin{corollary}
Let $\mathcal{S}_{0}$ be the operator system generated by the irreducible
periodic unilateral weighted shift with period $p$, and $\mathcal{I}%
_{0}=K\left( l^2(\mathbb{N}) \right)$ denote the Shilov boundary ideal for $%
\mathcal{S}_{0}$ (see \cite{ArFa13}), then $C\left( X, K\left( l^2(\mathbb{N}%
) \right) \right)$ is a boundary ideal for $C(X, \mathcal{S}_{0})$.
\end{corollary}

If $X$ is a compact Hausdorff space, then $C(X, \mathcal{S}_{0})$ is an
operator system contained in the $C^*$-algebra $C\left( X, C^*(\mathcal{S}%
_{0}) \right)$ of continuous functions from $X$ into the $C^*$-algebra
generated by $\mathcal{S}_{0}$. Moreover, $C\left( X, K\left( l^2(\mathbb{N}%
) \right) \right)$ is a boundary ideal for $C(X, \mathcal{S}_{0})$. However,
it is not known whether this ideal is the Shilov boundary ideal for $C\left(
X, \mathcal{S}_{0} \right)$. \medskip

\begin{example}
We apply Remark \ref{Rem} to the case where $\mathcal{S}_{0}$ is the
operator system generated by the irreducible periodic unilateral weighted
shift with period $p$. In this situation, the local operator system $%
\mathcal{S}:=\lbrace \left( s_{n} \right)_{n\in\mathbb{N}}\mid s_{n}\in%
\mathcal{S}_{0} \rbrace\subseteq C^*_{\mathcal{E}}(\mathcal{D})$, where $%
\mathcal{D}=\bigcup\limits_{n\in\mathbb{N}}l^{2}(\mathbb{N})^{\oplus n}$,
has Shilov boundary ideal $\lbrace \left( T_{n} \right)_{n\in\mathbb{N}}\mid
T_{n}\in K(l^{2}(\mathbb{N}))\rbrace$.
\end{example}

\subsection{$\Gamma$-boundary representations}

Let $\mathcal{S}$ be a Fr\'{e}chet local operator system and $\mathcal{A}$
be the locally $C^*$-algebra generated by $\mathcal{S}$. We denote by $%
\Gamma:=\lbrace p_{m} \rbrace_{m\geq 1}$ the family of $C^*$-seminorms that
define the topology on $\mathcal{A}$.

\begin{definition}[\protect\cite{Jo25}, Definition 2.1]
Let $\varphi :\mathcal{S}\rightarrow B(\mathcal{H})$ be a linear map. If
there exists $m\geq 1$ such that $\varphi ^{(n)}([a_{ij}]_{i,j=1}^{n})$ is
positive in $B(\mathcal{H}^{\oplus n})$ whenever $[a_{ij}]_{i,j=1}^{n}$ is $m
$-positive and $\varphi ^{(n)}([a_{ij}]_{i,j=1}^{n})=0$ whenever $%
[a_{ij}]_{i,j=1}^{n}=_{m}0$, for all $n\geq 1$, we say that $\varphi $ is a $%
m$-$\mathcal{CP}$ (completely positive) map.
\end{definition}

\begin{definition}[\protect\cite{Jo25}, Definition 3.9]
A linear map $\pi:\mathcal{A}\rightarrow B(\mathcal{H})$ is a $p_{m}$%
-boundary representation for $\mathcal{S}$ if

\begin{itemize}
\item[(1)] $\pi$ is an irreducible representation;

\item[(2)] $\pi $ is the unique unital $m$-$\mathcal{CP}$ map that extends $%
\pi \upharpoonright _{\mathcal{S}}.$
\end{itemize}
\end{definition}

\begin{construction}
\label{pi_m} Let $\lbrace \mathcal{H};\mathcal{E};\mathcal{D} \rbrace$ be a
Fr\'{e}chet quantized domain and let $\varphi:\mathcal{A}\rightarrow C^*_{%
\mathcal{E}}(\mathcal{D})$ be a linear map. We define the maps 
\begin{equation*}
\varphi_{m}:\mathcal{A}\rightarrow B\left( \mathcal{H}_{m-1}^{\perp}\cap%
\mathcal{H}_{m}\right),\ \varphi_{m}(a):=\varphi(a)\restriction_{\mathcal{H}%
_{m-1}^{\perp}\cap\mathcal{H}_{m}} \ \mathit{for}\ m>1
\end{equation*}
and 
\begin{equation*}
\varphi_{1}:\mathcal{A}\rightarrow B(\mathcal{H}_{1}),\
\varphi_{1}(a):=\varphi(a)\restriction_{\mathcal{H}_{1}}.
\end{equation*}
\end{construction}

We denote by $\mathcal{K}_{1}:=\mathcal{H}_{1}$ and $\mathcal{K}_{m}:=%
\mathcal{H}_{m-1}^{\perp}\cap\mathcal{H}_{m}\ (m\geq 2)$ so that $\mathcal{H}%
=\overline{\bigoplus\limits_{m\geq 1}\mathcal{K}_{m}}$. \medskip

If $\varphi :\mathcal{A}\rightarrow C_{\mathcal{E}}^{\ast }(\mathcal{D})$ is
a local completly positive map, then the maps $\varphi _{m}:\mathcal{A}%
\rightarrow B(\mathcal{K}_{m}),\ m\geq 1$ are completely positive (see \cite[%
Remark 4.6]{Simon}).

If $\pi:\mathcal{A}\rightarrow C^*_{\mathcal{E}}(\mathcal{D})$ is a local
representation, then the maps $\pi_{m}:\mathcal{A}\rightarrow B(\mathcal{K}%
_{m}), \ m\geq 1$ are representations of $\mathcal{A}$ on Hilbert spaces $%
\mathcal{K}_{m}.$

The \textit{center} $\mathcal{Z}(C^*_{\mathcal{E}}(\mathcal{D}))$ of $C^*_{%
\mathcal{E}}(\mathcal{D})$ is the SOT-closure of the unital algebra
generated by the family of projections $\lbrace P_{\iota}
\rbrace_{\iota\in\Omega}$ (see \cite[Corollary 3.2]{Do12}).

\begin{definition}[\protect\cite{Simon}, Definition 3.8]
\label{ired} Let $\lbrace \mathcal{H}; \mathcal{E}; \mathcal{D} \rbrace$ be
a quantized (commutative) domain and let $\pi:\mathcal{A}\rightarrow C^*_{%
\mathcal{E}}(\mathcal{D})$ be a local representation of $\mathcal{A}$. We
say that $\pi:\mathcal{A}\rightarrow C^*_{\mathcal{E}}(\mathcal{D})$ is
irreducible if $\pi(\mathcal{A})^{\prime }=\mathcal{Z}\left( C^*_{\mathcal{E}%
}(\mathcal{D})\right)$, where $\pi(\mathcal{A})^{\prime *}_{\mathcal{E}}(%
\mathcal{D})\mid T\pi(a)=\pi(a)T \rbrace$ stands for the commutant of $\pi(%
\mathcal{A})$ in $C^*_{\mathcal{E}}(\mathcal{D}).$
\end{definition}

The collection of all $p_{m}$-boundary representations for $\mathcal{S}$ is
denoted by $Ch_{m}(\mathcal{S}).$

\begin{definition}
\label{Gama rep.} We say that a local representation $\pi :\mathcal{A}%
\rightarrow C_{\mathcal{E}}^{\ast }(\mathcal{D})$ for $\mathcal{S}$ is a 
\textit{$\Gamma $-boundary representation for $\mathcal{S}$} if for each $%
m\geq 1$, $\pi _{m}$ is a $p_{m}$-boundary representation for $\mathcal{S}$.
\end{definition}

We denote by $Ch_{\Gamma}(\mathcal{S})$ the set of all $\Gamma$-boundary
representations for $\mathcal{S}$. \medskip

\begin{remark}
If $\pi :\mathcal{A}\rightarrow C_{\mathcal{E}}^{\ast }(\mathcal{D})$ is a
local boundary representation for $\mathcal{S}$, then the following
statements do hold:

\begin{itemize}
\item[i)] For each $m\geq 1,$ by \cite[Theorem 3.14]{Simon}, $\pi _{m}$ is
an irreducible representation of $\mathcal{A}$

\item[ii)] If for each $m\geq 1,\varphi _{m}$ is a an $m$-$\mathcal{CP}$ map
which extends $\pi _{m}\upharpoonright _{\mathcal{S}}$, then the map $%
\varphi :$ $\mathcal{A}\rightarrow C_{\mathcal{E}}^{\ast }(\mathcal{D}),$
defined by $\varphi \left( a\right) \upharpoonright _{\mathcal{K}%
_{m}}=\varphi _{m}\left( a\right) $ for all $a\in \mathcal{A},$ is a local
completely positive map such that $\varphi \upharpoonright _{\mathcal{S}%
}=\pi \upharpoonright _{\mathcal{S}}.$ Consequently, since $\pi $ is a local
boundary representation, $\varphi =\pi .$ Therefore, $\varphi _{m}=\pi _{m}.$

\item[iii)] $\pi $ is a $\Gamma $-boundary representations for $\mathcal{S}$.
\end{itemize}
\end{remark}

Let $\{\pi _{m}\}_{m\geq 1}$ be a sequence of $p_{m}$-boundary
representations for $\mathcal{S}$, $\pi _{m}:\mathcal{A}\rightarrow B(%
\widetilde{\mathcal{K}}_{m})$. For each $m\geq 1$, let $\mathcal{H}_{m}=%
\widetilde{\mathcal{K}}_{1}\oplus \widetilde{\mathcal{K}}_{2}\oplus \ldots
\oplus \widetilde{\mathcal{K}}_{m}$. Therefore, $\{\mathcal{H};\mathcal{E};%
\mathcal{D}\}$ is a Fr\'{e}chet quantized domain, where $\mathcal{E}=\{%
\mathcal{H}_{m}\}_{m\geq 1},\ \mathcal{D}=\bigcup\limits_{m\geq 1}\mathcal{H}%
_{m}$ and $\mathcal{H}=\overline{\bigcup\limits_{m\geq 1}\mathcal{H}_{m}}=%
\overline{\mathcal{D}}$.

\begin{proposition}
\label{RR} There exists a bijective correspondence between $Ch_{\Gamma}(S)$
and $\lbrace (\sigma_{m})_{m\geq 1}\mid \sigma_{m}\in Ch_{m}(S) \rbrace$. 
\end{proposition}

\begin{proof}
Let $\pi\in Ch_{\Gamma}(\mathcal{S})$. Then, for each $m\geq 1$, the map $%
\pi_{m}$ associated to $\pi$ by Construction \ref{pi_m}, is a $p_{m}$%
-boundary representation for $\mathcal{S}$. Thus, we have a map $\Phi:
Ch_{\Gamma}(\mathcal{S})\rightarrow \lbrace (\sigma_{m})_{m\geq 1}\mid
\sigma_{m}\in Ch_{m}(S) \rbrace $, $\Phi(\pi)=\lbrace \pi_{m} \rbrace_{m\geq
1}.$

Let $\{\sigma _{m}\}_{m\geq 1}$ be a family of $p_{m}$-boundary
representations for $\mathcal{S}$. Then the map $\pi :\mathcal{A}\rightarrow
C_{\mathcal{E}}^{\ast }(\mathcal{D})$ defined by $\pi (a)\upharpoonright _{%
\mathcal{K}_{m}}=\sigma _{m}(a),(\forall )\ a\in \mathcal{A},(\forall )\
m\geq 1$ is a $\Gamma $-boundary representation for $\mathcal{S}$. Hence, $%
\pi _{m}=\sigma _{m},(\forall )\ m\geq 1$, and so, $\Phi (\pi )=\{\sigma
_{m}\}_{m\geq 1}$. Moreover, $\pi $ is unique with this property.
\end{proof}

In the following result, we show that $\Gamma$-boundary representations are
intrinsic invariants of Fr\'{e}chet local operator systems, thereby
generalizing the corresponding result for operator systems (see \cite[%
Theorem 4.2]{Ar69}).

\begin{proposition}
\label{P1} Let $\mathcal{A}_{1}$ and $\mathcal{A}_{2}$ be two unital locally 
$C^{\ast }$-algebras with the topologies given by the families of $C^{\ast }$%
-seminorms $\Gamma _{1}:=\{p_{m}\}_{m\geq 1}$ and $\Gamma
_{2}:=\{q_{m}\}_{m\geq 1}$, respectively. Let $\mathcal{S}_{1}\subseteq 
\mathcal{A}_{1}$ and $\mathcal{S}_{2}\subseteq \mathcal{A}_{2}$ be two Fr%
\'{e}chet local operator systems such that $\mathcal{S}_{1}$ generates $%
\mathcal{A}_{1}$ and $\mathcal{S}_{2}$ generates $\mathcal{A}_{2}$,
respectively. If $\varphi :\mathcal{S}_{1}\rightarrow \mathcal{S}_{2}$ is a
unital surjective local completely isometric map, then for each $\Gamma _{1}$%
-boundary representation $\pi _{1}:\mathcal{A}_{1}\rightarrow C_{\mathcal{E}%
}^{\ast }(\mathcal{D})$ for $\mathcal{S}_{1}$, there exists a $\Gamma _{2}$%
-boundary representation $\pi _{2}:\mathcal{A}_{2}\rightarrow C_{\mathcal{E}%
}^{\ast }(\mathcal{D})$ for $\mathcal{S}_{2}$ such that $\Phi _{2}\circ
\varphi =\Phi _{1}\upharpoonright _{\mathcal{S}_{1}}.$
\end{proposition}

\begin{proof}
Let $\pi _{1}:\mathcal{A}_{1}\rightarrow C_{\mathcal{E}}^{\ast }(\mathcal{D}%
) $ be a $\Gamma _{1}$-boundary representation for $\mathcal{S}_{1}$. Then
for all $m\geq 1$, $\left( \pi _{1}\right) _{m}:\mathcal{A}_{1}\rightarrow B(%
\mathcal{K}_{m})$ is a $p_{m}$-boundary representation for $\mathcal{S}_{1}$
on the Hilbert space $\mathcal{K}_{m}.$ By \cite[Proposition 3.1]{Jo25},
there exists a $q_{m}$-boundary representation $\left( \pi _{2}\right) _{m}:%
\mathcal{A}_{2}\rightarrow B(\mathcal{K}_{m})$ for $\mathcal{S}_{2}$ on
Hilbert space $\mathcal{K}_{m}$ such that $\left( \pi _{2}\right) _{m}\circ
\varphi =\left( \pi _{1}\right) _{m}\upharpoonright _{\mathcal{S}_{1}}.$ By
the proof of \cite[Proposition 5.5]{Simon}, it follows that there exists a $%
\Gamma _{2}$-boundary representation $\pi _{2}:\mathcal{A}_{2}\rightarrow C_{%
\mathcal{E}}^{\ast }(\mathcal{D})$ for $\mathcal{S}_{2}$, defined by 
\begin{equation*}
\pi _{2}(a)\upharpoonright _{\mathcal{H}_{m}}:=\bigoplus\limits_{k\leq
m}\left( \pi _{2}\right) _{k}(a),\ (\forall )\ a\in \mathcal{A}_{2},\
(\forall )\ m\geq 1.
\end{equation*}%
Therefore, it is clearly that $\pi _{2}\circ \varphi =\pi
_{1}\upharpoonright _{\mathcal{S}_{1}}$.
\end{proof}

W. Arveson \cite{Ar08} proved that for every separable operator system $%
\mathcal{S}$, the Shilov boundary ideal of $\mathcal{S}$ is given by the
intersection of the kernels of all boundary representations of $\mathcal{S}$%
. In \cite[Theorem 3.16]{Jo25}, by analogy with the normed case, it was
shown that the Shilov boundary ideal for a separable local operator system
can be characterized as the intersection of the kernels of all $p_{\lambda}$%
-boundary representations of $\mathcal{S}$. In the following, we show that,
for a unital separable Fr\'{e}chet local operator system $\mathcal{S}$, the
Shilov boundary ideal coincides with the intersection of the kernels of all $%
\Gamma$-boundary representations of $S$.

\begin{lemma}
\label{L1} Let $\pi:\mathcal{A}\rightarrow C^*_{\mathcal{E}}(\mathcal{D})$
be a local representation. Then 
\begin{equation*}
\ker\pi=\bigcap\limits_{m\geq 1}\ker\pi_{m}.
\end{equation*}
\end{lemma}

\begin{proof}
If $a\in\ker\pi$, then $\pi(a)=0$ on $\mathcal{D}$. Hence $%
\pi_{m}(a)=\pi(a)\restriction_{\mathcal{K}_{m}}=0$ for every $m\geq 1$.
Therefore $a\in\ker\pi_{m}$, for every $m\geq 1.$

Conversely, if $a\in \bigcap\limits_{m\geq 1}\ker\pi_{m}$, then $%
\pi_{m}(a)=0 $ for every $m\geq 1$. Since $\pi(a)$ is determined by the
family of representations $\lbrace \pi_{m}(a) \rbrace_{m\geq 1}$ on the
Hilbert space $\mathcal{K}_{m}$, it follows that $\pi(a)=0.$ Therefore $%
a\in\ker\pi.$
\end{proof}

\begin{theorem}
\label{Theorem} Let $\mathcal{A}$ be a unital separable Fr\'{e}chet locally $%
C^*$-algebra whose topology is given by the family of $C^*$-seminorms $%
\lbrace p_{m} \rbrace_{m\geq 1}$ and let $\mathcal{S}\subseteq\mathcal{A}$
be a local operator system which generates $\mathcal{A}$. Then 
\begin{equation*}
\mathcal{J}:=\bigcap\limits_{\pi\in Ch_{\Gamma}(\mathcal{S})}\ker\pi
\end{equation*}
is the Shilov boundary ideal for $\mathcal{S}$.
\end{theorem}

\begin{proof}

By Lemma \ref{L1} and Remark \ref{RR}, (ii) we obtain the following
relation: 
\begin{equation}  \label{relation 1}
\bigcap\limits_{\pi\in Ch_{\Gamma}(\mathcal{S})}\ker\pi=\bigcap\limits_{m%
\geq 1}\bigcap\limits_{\pi_{m}\in Ch_{m}(\mathcal{S})}\ker\pi_{m}.
\end{equation}

On the other hand, by \cite[Theorem 3.16]{Jo25}, $\bigcap\limits_{m\geq
1}\bigcap\limits_{\pi_{m}\in Ch_{m}(\mathcal{S})}\ker\pi_{m}$ is the Shilov
boundary ideal for $\mathcal{S}$. Consequently, $\mathcal{J}$ is the Shilov
boundary ideal for $\mathcal{S}$.
\end{proof}

\subsection*{Funding}

The authors declare that no funds, grants, or other support were received
during the preparation of this manuscript.

\subsection*{Data Availability}

This paper has no associated data.

\subsection*{Declarations}

\subsection*{Conflicts of interests}

The author has no relevant financial or non-financial interests to disclose.

\end{document}